\documentclass[11pt,a4paper]{article}

\usepackage[T1]{fontenc}
\usepackage[applemac]{inputenc}
\usepackage[english]{babel}

\usepackage[all]{xy}
\usepackage{amsmath}
\usepackage{amssymb}
\usepackage{amsthm}
\usepackage{enumerate}

\usepackage{graphicx}
\usepackage{psfrag}

\usepackage{geometry}

\newtheorem{thm}{Theorem}[section]
\newtheorem{defn}[thm]{Definition}
\newtheorem{prop}[thm]{Proposition}
\newtheorem{lemma}[thm]{Lemma}

\newcommand{\Hh}{{\mathcal H}}

\newcommand{\Oo}{{\mathcal O}}

\newcommand{\Tt}{{\mathcal T}}

\newcommand{\NM}{{\mathbb N}}

\newcommand{\RM}{{\mathbb R}}

\renewcommand{\phi}{\varphi}                  %% nice phi  
\renewcommand{\tilde}{\widetilde}

\newcommand{\wee}{\, \tilde\wedge \,}       % branching vertex in tree of return words

\newcommand{\dsup}{d_{\text{\rm sup}}}   % upper spectral distance
\newcommand{\dinf}{d_{\text{\rm inf}}}   % lower spectral distance
\newcommand{\op}{\text{\rm op}}     % opposit (for edges)  
\newcommand{\rp}{r_{\text{\rm pack}}}
\newcommand{\rc}{r_{\text{\rm cov}}}
\newcommand{\rmin}{r_{\text{min}}}

\title{A metric characterisation of repulsive tilings}
\author{J. Savinien$^{1,2}$\\
{\small $^{1}$ Université de Lorraine, Institut Elie Cartan de Lorraine, UMR 7502, Metz, F-57045, France.}\\
{\small $^{2}$ CNRS, Institut Elie Cartan de Lorraine, UMR 7502, Metz, F-57045, France.}
}

\begin{document}

\maketitle

\begin{abstract}
A tiling of $\RM^d$ is {\it repulsive} if no $r$-patch can repeat arbitrarily close to itself, relative to $r$.
This is a characteristic property of aperiodic order, for a non repulsive tiling has arbitrarily large local periodic patterns.

We consider an aperiodic, repetitive tiling $T$ of $\RM^d$, with finite local complexity.
From a spectral triple built on the discrete hull $\Xi$ of $T$, and its Connes distance, we derive two metrics $\dsup$ and $\dinf$ on $\Xi$.
We show that $T$ is repulsive if and only if $\dsup$ and $\dinf$ are Lipschitz equivalent.
This generalises previous works for subshifts by J. Kellendonk, D. Lenz, and the author. %which are symbolic one-dimensional tilings.
\end{abstract}

%%%%%%%%%%%%%%%%%%%%%%%%%%%%%%%%%%%%%%%%%%%%%%%%%%%%%%%%%
\section{Introduction}

In two recent articles in collaboration with J. Kellendonk and D. Lenz  \cite{KS10,KLS11}, 
we used constructions of non commutative geometry \cite{Co94} to derive a new characterisation of  aperiodically ordered $1d$-subshifts.
We showed that a minimal and aperiodic subshift $X$ has bounded powers if and only if two metrics derived from the Connes distance of a spectral triple over $X$ are Lipschitz equivalent.
An essential ingredient to obtain this result is the notion of {\em privileged words} \cite{KLS11}.
In this paper, we generalise this formalism and this results to tilings of $\RM^d$.
The essential ingredient here is the notion of {\em privileged patches} of a tiling.

A $1d$-subshift has bounded powers if its language does not contain arbitrarily large powers, {\it i.e.} there is an integer $p$ such that $n$-fold
repetitions $w^n=w\cdots w$ of a word $w$ cannot occur for $n>p$.
Linearly recurrent subshifts, which are usually considered highly ordered, share this property \cite{LP03, Dur00, Dur03}. 
Loosely speaking, bounded powers means that no factor can repeat too close, or overlap too much, along a sequence in the subshift.
Bounded powers is equivalent to the property that a complete first return $u'$ of a word $u$ must be longer than a uniform constant times the length of $u$: $|u'|> C |u|$.
The corresponding notion for tilings is {\em repulsiveness}: no patch can repeat arbitrarily close to itself relative to its size, see equation~\eqref{eq-rep}.
A non repulsive tiling has arbitrarily large local periodic patterns -- the analogue of arbitrarily large powers.
As for subshifts, linearly repetitive tilings are repulsive \cite{Pat98, Len04}.

The property of bounded (or unbounded) powers in a subshift is measured by privileged words.
Privileged words are iterated complete first returns to letters of the alphabet. 
Privileged words were introduced in \cite{KLS11}, and have recently encountered a lot of interest in the combinatorics of words \cite{Pelto13, FPZ13, Pelto, FJS13}.
For rich subshifts \cite{GJWZ09} privileged words coincide exactly with palindromes (see \cite{KLS11} Section 2.2 for further details).

We generalise this notion to tilings. 
We define privileged patches: a notion of iterated complete first returns to the prototiles, see Section~\ref{sec-priv}.
For $1d$ subshifts, a privileged patch is a generalisation of a privileged word obtained with bilateral versions of complete first returns.
Because of the geometry in $\RM^d$, the combinatorics of patches is much more involved than that of words.
We need a few technical lemmas to deal with this.
But the crucial point is the generalisation of privileged words to the tilings setting.
Once the right definition of privileged patch is at hand, our formalism %of spectral triples and Lipschitz equivalence of Connes distances 
for subshifts essentially goes through for tilings of $\RM^d$.
The spectral triple we used in \cite{KLS11} for subshift is build from the tree of privileged words of the subshift.
The spectral triple we use here is the same one built on the tree of privileged patches of the tiling.
This allows us to characterise repulsive tilings by Lipschitz equivalence of two metrics derived from the Connes distance of the spectral triple, in complete analogy with the case of subshifts treated in \cite{KLS11}.

Our initial motivation in studying properties of aperiodically ordered subshifts and tilings, came from non commutative geometry (NCG) \cite{Co94}.
Namely we were interested in the construction of non commutative Riemanian structures, {\it i.e.} spectral triples, over totally disconnected spaces defined by tilings and subshifts.
As it turns out, and as in~\cite{KLS11}, the criterium for aperiodic order we derive here can be explained in a rather combinatorial way, without  introducing the framework of NCG and giving the details of the construction of the spectral triple.
So we follow this line in the paper: we give the criterium {\it ad hoc} to state and prove our result.
And in the last section we describe briefly the underlying spectral triple.

\medskip

The paper is organised as follows.
In Section~\ref{sec-basics} we remind the reader of the basic definitions for tilings of $\RM^d$, and the classical results we need.
We introduce privileged patches in Section~\ref{sec-priv}, and state some combinatorial properties, including technical lemmas which allows us to adapt our formalism for subshifts to tilings of $\RM^d$.
In Section~\ref{sec-tree} we explain the construction of the tree of privileged patches, from which we define the two Connes metrics.
In Section~\ref{sec-charact} we state and prove our main result, namely that a tiling is repulsive if and only if the Connes metrics are Lipschitz equivalent.
The construction of the spectral triple, from which the Connes metrics are derived, is given briefly in Section~\ref{sec-ST}.

\medskip

\noindent {\bf Aknowlegements.} The author would like to thank J. Kellendonk and D. Lenz for useful discussions, and encouragements to publish this work.

%%%%%%%%%%%%%%%%%%%%%%%%%%%%%%%%%%%%%%%%%%%%%%%%%%%%%%%%%
\section{Basic definitions}
\label{sec-basics}

A {\em tile} of $\RM^d$ is a subset $t \subset \RM^d$ which is homeomorphic to a closed ball.
A {\em tiling} of $\RM^d$, is a countable family of tiles, $T=\{ t_i\}_{i\in\NM}$, which have pairwise disjoint interiors and whose union covers $\RM^d$.
Given a tiling $T$, we specify a {\em marker}~\footnote{sometimes also called a {\em puncture}, so that one talks about {\em punctured tilings}.} in each of its tile $t$: a point $x(t)\in \RM^d$ in its interior.

A {\em translate} of a family $F=\{ t_j\}_{j\in J}$ of tiles of $T$, is a family $F+a=\{t_j + a\}_{j\in J}$, for some $a\in\RM^d$.
Let $x$ be the marker of a tile of $T$, and $r>0$. 
We call an {\em $r$-patch}, or a {\em patch} of radius $r$, the finite family of tiles of $T-x$ all of whose markers lie inside the open ball $B(0,r)$. 
In addition, $r$ is maximal with respect to the family of tiles defining the patch.
As a consequence, the only $0$-patch is the empty patch.
%Notice that a patch is marked at the origin.
The patches made of a single tile (containing the marker of a single tile), are called {\em prototiles}.

Consider an $r$-patch $p$ of $T$.
Given a family $F=\{ t_j\}_{j\in J}$ of tiles of $T$, we say that {\em $p$ occurs in $F$}, if there is a translate of $p$ which is a subset of $F$: $p+a\subset F$ for some $a \in \RM^d$.
The translate $p+a$ is called an {\em occurrence} of $p$ in $F$.
Given a subset $U$ of $\RM^d$, we say that {\em $p$ occurs in $U$}, if there is an occurrence of $p$ in $T$, the union of all of whose tiles is a subset of $U$. 
We mean that a patch $p$ is marked at the origin: $x(p)=0$.
And that an occurrence of $p$ in $T$, in a family of tiles $F$, or in a subset of $\RM^d$, is some translated copy $p+a$ marked at $a$: $x(p+a)=a$. 

%[NEEDED?]

%In both definitions, the translate of $p$ is called an {\em occurrence} of $p$ in $F$ or $U$.
%We write it $p+a$, for some $a\in\RM^d$, and understand it as a copy of $p$ with marker $a$.

We will consider tilings satisfying the following three properties.
%%%
\begin{defn}
\label{def-hypT}
A tiling $T$ of $\RM^d$ is called
\begin{enumerate}[(i)]

\item {\em aperiodic} if $T+a=T$ implies $a=0$;

\item {\em repetitive} if for any $r>0$, and any $r$-patch $p$ of $T$, there exists $R>0$ such that $p$ occurs in any ball of radius $R$;

\item {\em FLC}, or has {\em Finite Local Complexity}, if for any $r>0$ there are finitely many $r$-patches.

\end{enumerate}
\end{defn}
%%%
%In properties (ii) and (iii), we understand that translations carry markers over: a translate of a patch $p$ has marker $x(p+a)=x(p)+a$ (and similarly for the markers of its tiles).

Let $T$ be a repetitive and FLC tiling of $\RM^d$, and $p$ an $r$-patch of $T$.
The {\em Delone set of occurrences of $p$ in $T$} is the set $L_p$ of markers of all occurrences of $p$ in $T$.
This is a Delone set as the distance between nearest points of $L_p$ is uniformly bounded.
We let $\rp(L_p)$ (resp. $\rc(L_p)$) be one half of that uniform minimal distance (resp. maximal distance). 
It is called the {\em packing radius}  of $L_p$ (resp. {\em covering radius}): any ball of radius $\rp(L_p)$ (resp. $\rc(L_p)$) contains at most (resp. at least) one point of $L_p$.
A tiling $T$ is said to be {\em repulsive} if
%%%
\begin{equation}
\label{eq-rep}
 \ell = \inf \Bigl\{ \frac{\rp(L_p)}{r} \, : \, \text{\rm $p$ an $r$-patch of $T$}  \Bigr\} > 0.
\end{equation}
%%%
Informally, $\ell>0$ means that patches cannot overlap too much.
On the contrary, in a non-repulsive tiling, there are arbitrarily large $r$-patches with arbitrarily close occurrences relative to $r$.
Such occurrences overlap over an arbitrarily large proportion of their tiles.
This implies that a non-repulsive tiling has arbitrarily large local periodic patterns, see Figure~\ref{fig-nrep}.
%And the overlappings of such patches creating arbitrarily large local periodic patterns.
 
\medskip

We now fix an aperiodic, repetitive, and FLC tiling $T$ of $\RM^d$ and assume that there is a tile whose marker lies at the origin.
We endow the family of all of its translates, $T+\RM^d$, with the following topology.
A base of open sets is given by the acceptance domains of patches: for $p$ an $r$-patch of $T$
\[
[p] = \bigl\{ T'\in \ T+\RM^d \, : \, p \ \text{\rm occurs at the origin in } T' \bigr\}
\] 
If $q$ is a patch contained in $p$, which we write $q\subseteq p$, then $[p]\subset [q]$.
Hence two tilings are close for this topology, if they agree on a large patch around the origin.
%the set of tilings for which $p$ occurs at the origin.
The {\em discrete hull} of a $T$ is the closure of its translates in this topology:
\[
\Xi = \overline{ \bigl\{ T + a \, :  \, a \in \RM^d, \, T + a \, \text{\rm has a marker at the origin} \bigr\}}\,.
\]
As a consequence of the hypothesis in Definition~\ref{def-hypT} the following classical results hold:
\begin{enumerate}[-]

\item $\Xi$ is a Cantor set (compact, totally disconnected, with no isolated point);

\item the family of acceptance domains $[p]$ is a countable base of clopen sets\footnote{closed and open sets} for $\Xi$;

\item any $T'\in \Xi$ satisfies the hypothesis of Definition~\ref{def-hypT}, and the closure of $T'+\RM^d$ is $\Xi$.

\end{enumerate}
%A basis for the topology is given by the acceptance domains of patches: $[p]$ is the set of tilings in $\Xi$ where $p$ occurs at the origin (with marker $x(p)=0$).
The discrete hull is metrizable.
Any function $\delta: [0,+\infty) \rightarrow (0,1]$, which decreases and has limit $0$ at $+\infty$, defines a ultra-metric on $\Xi$ as follows: 
%the distance between $T_1$ and $T_2$ is set to be $\delta(r_{12})$, where $r_{12}$ is the radius of the largest $r$-patch occuring at the origin in both $T_1$ and $T_2$.
\begin{equation}
\label{eq-ultrametric}
d_\delta (T_1,T_2) = 
\inf \bigl\{ \delta(r) : \text{\rm there is an $r$-patch $p$ occurring in both $T_1$ and $T_2$ at the origin} \bigr\}.
\end{equation}

%%%%%%%%%%%%%%%%%%%%%%%%%%%%%%%%%%%%%%%%%%%%%%%%%%%%%%%%%
\section{Privileged patches}
\label{sec-priv}

%As a consequence of FLC, we can consider the increasing sequence $(r_n)_\NM$ of radii of  the $r$-patches of $T$. 
%We set $r_0=0$ and $\rmin=r_1$, the respective radii of the empty patch and the smallest tile. [NEEDED?]

Given an $r$-patch $p$, we say that an $r'$-patch $p'$ is {\em derived} from $p$ if
%\marginpar{define occurrence number?}
%%%
\begin{enumerate}[(i)]
\item $p$ is contained in $p'$; %and $r<r'$,
%$p'\cap B(x(p'),r(p))=p$;
\item $p$ occurs at least twice in $p'$;
\item for any $\tilde{r}<r'$, and any $\tilde{r}$-patch $q$ contained in $p'$, $p$ occurs at most once in $q$.
%the restriction $p'\cap B(x(p'),r_n)$ contains at most one copy of $p$.
\end{enumerate}
%%%
See Figure~\ref{fig-derpatch} for an illustration.
Condition (ii) means that $p'$ contains $p$ as a subpatch (hence with marker at the origin), plus another translate $p+a$ for some $a\neq 0$.
Conditions (ii) and (iii) mean that $p'$ is a minimal extension of $p$ containing two occurrences of $p$.
%above implies that the number of occurrences of $p$ in $p'$ is at least $2$.
%Let us denote by $\rp(L_p)$ and $\rc(L_p)$ the packing and covering radii of  $L_p$ (the Delone set of occurrences of $p$).
%%%
\begin{figure}[h]
\centering
\includegraphics[scale=.5]{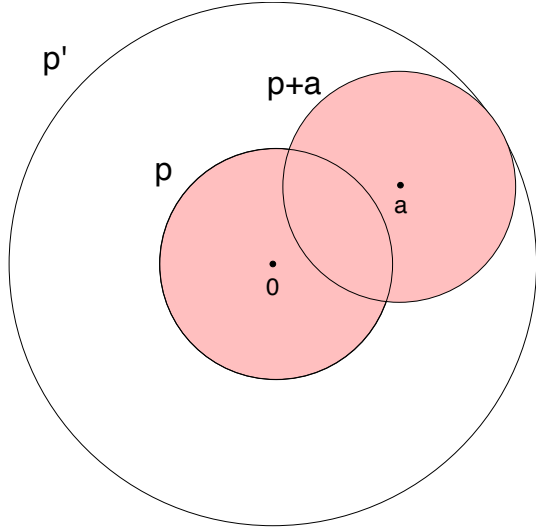}
\caption{A patch $p'$ derived from $p$.}
\label{fig-derpatch}
\end{figure}
%%%

We define {\em privileged patches} inductively, as follows:
%%%
\begin{enumerate}
\item[(0)] the empty patch is the only privileged patch of order $0$;
\item[(1)] the prototiles of $T$, are the privileged patches of order $1$;
\item[(n)] for $n>1$ a privileged patch of order $n$ is an $n$-th iterated derived patch from the empty patch.
\end{enumerate}
%%%
For  $1d$-subshifts, {\it i.e.} symbolic one-dimensional tilings, this is a two-sided version of {\em privileged words} introduced in \cite{KLS11}.
%\marginpar{notation $p^{(n)}$?}

\bigskip

Let us state some elementary properties of derived patches.
The first two Lemmas are needed to build the tree of privileged patches in the next Section.
%%%%%%%
\begin{lemma}
\label{lem-der}
Let $q$ be a patch derived from some patch, then
%%%
\begin{enumerate}[(i)]
\item there exists a unique patch $p$ such that $p'=q$;

\item if $q$ is privileged, then there exists a unique privileged patch $p$ such that $p'=q$;

\item if $q$ is privileged, and $p$ is a privileged patch contained in $q$, then there exists $i\ge 0$ such that $q$ is an $i$-th iterated derived patch from $p$, which we write $p^{(i)}=q$.

\end{enumerate}
%%%
\end{lemma}
%%%%%%%
\begin{proof}
(i) Assume that $q=p'_1=p'_2$, for two distincts patches $p_i$ of radius $r_i$, $i=1,2$.
We may assume $r_2<r_1$, but then $p_2 \subsetneq p_1$, and this implies $p'_2 \subsetneq p'_1$ a contradiction. 

(ii) If $q$ is privileged, by definition there exists a privileged patch $p$ such that $p'=q$, and by (i) $p$ is unique.

(iii) We prove this by induction on the radius of $q$. %(the increasing sequence of radii of patches of $T$).
Let $(r_n)_\NM$ be the non-decreasing sequence of radii of privileged patches of $T$ (which exists by FLC).  
The property is obvious for privileged patches of radius $r_1$: $q$ is a prototile with smallest radius %($\tilde{r}_1=r_1=\rmin$) 
and is derived for the empty patch.
Assume the property holds for all privileged patches of radii less than or equal to $r_{n}$, for some $n>1$.
Consider a privileged patch $q$ of radius $r_{n+1}$, and a privileged patch $p\subseteq q$.
The case $p=q=p^{(0)}$ is trivial, so assume $p\subsetneq q$.
By (ii) there exists a unique privileged patch $\tilde p\subsetneq q$ with $\tilde p^{(1)}=q$.
Case $p\subseteq\tilde p$ : by induction $\tilde p = p^{(j)}$ for some $j\ge 0$, and so $q=p^{(j+1)}$.
Case $\tilde p\subsetneq p$ : by induction there is a $j>0$ such that $p=\tilde p^{(j)} = q^{(j-1)}$, which implies $q\subseteq p$ a contradiction.
\end{proof}

%%%
\begin{lemma}
\label{lem-rder}
\begin{enumerate}[(i)]
\item Let $p$ be an $r$-patch, and $p'$ an $r'$-patch derived from $p$, then 
\[
2 \rp( L_p)  + r \le r' \le 2 \rc( L_p)  + r.
\]

\item Let $(p_n)_{n\ge1}$ be a sequence of $r_n$-patches, such that $p_{n+1}$ is derived from $p_n$ for all $n$.
Then 
\( r_{n+1} \ge 2n \,\rmin\), where $\rmin$ is the radius of the smallest prototile.
If in addition $T$ is repulsive, then \(r_{n+1}\ge (2\ell + 1)^n \rmin\).
\end{enumerate}
\end{lemma}
%%%
\begin{proof}
The first claim follows at once from the definition.
We use the first inequality in (i) inductively to get
%%%
\[
r_{n+1} \ge 2 \rp( L_{p_{n}}) + r_{n} \ge 2 \, \rmin + r_{n} \ge \cdots \ge 2n \, \rmin.
\]
%%%
If $T$ is repulsive, then $\rp( L_{p_{j}}) \ge \ell r_j$ for all $j$, see equation~\eqref{eq-rep}, so one gets
\[
r_{n+1} \ge 2 \rp( L_{p_{n}}) + r_{n} \ge (2 \ell +1 ) r_{n} \ge \cdots \ge (2\ell +1)^n \rmin.
\]
\end{proof}

The following technical lemma is analogous to Lemma~3.8 in \cite{KLS11}.
It states that if a tiling is not repulsive, then one can find arbitrary long sequences of derived privileged patches whose radii grow ``slowly''.
%%%%%%%
\begin{lemma}
\label{lem-nrep}
If $T$ is not repulsive, then for all $m\in\NM$, there exists privileged patches $p_0, p_1,\ldots p_m$, of radii $r_0, r_1, \ldots r_m$ respectively, such that
%%%
\begin{enumerate}[(i)]
\item $p_{j}=p_0^{(j)}$, for all $j=1, \ldots, m$,

\item $r_m \le 2 r_1$.
%$r_{p_1} \ge \frac{1}{2} r_{p_m}$.

\end{enumerate}
%%%
\end{lemma}
%%%%%%%
\begin{proof}
Consider a non repulsive tiling $T$, and fix an integer $m>1$.
Since the infimum in equation~\eqref{eq-rep} is zero, for any $0<\epsilon <1/(8m)$, there is an $r_q$-patch $q$ of $T$ for which $\rp(L_q) / r_q < \epsilon$.
By FLC there are two copies of $q$ which occur at some markers $x$ and $y$ of tiles in $T$, satisfying $|x-y|=2\rp(L_q)$.

Set $a=\rp(L_q)$. % so we have $a<\epsilon r_q < r_q/4$.
Consider the largest privileged patch $p$ contained in $q$, with same marker, and of radius $r\le r_q$. %: $p \subseteq q$.
We must have $r\ge r_q/2$ for otherwise, as $p$ occurs both at $x$ and $y$, then one of its derived patches $p'$ would have radius \(r' \le r+ 2a < r_q/2 + 2\epsilon r_q < r_q\).
Hence $p'$ would be contained in $q$,
%which implies $p'\subsetneq q$ 
which contradicts maximality of $p$ in $q$.

Consider the largest privileged patch $p_0\subset p$, with same marker, and radius $r_0 < r/2$.
Since $p_0$ occurs both at $x$ and $y$, it has a derived patch $p_1\subset p$, with radius $r_1 \le r_0 + 2a < r/2 + \epsilon r_q<r$, which is thus a proper sub-patch of $p$.
By Lemma~\ref{lem-der} (iii), we have $p=p_1^{(n-1)}=p_0^{(n)}$ for some $n>1$.
Again $p_1$ occurs both at $x$ and $y$, so it has a derived patch $p_2=p_0^{(2)} \subsetneq p$ of radius $r_2\le r_1+2a \le r_0 + 4a$.
See Figure~\ref{fig-nrep} for an illustration.
%%%
\begin{figure}[h]
\centering
\includegraphics[scale=.5]{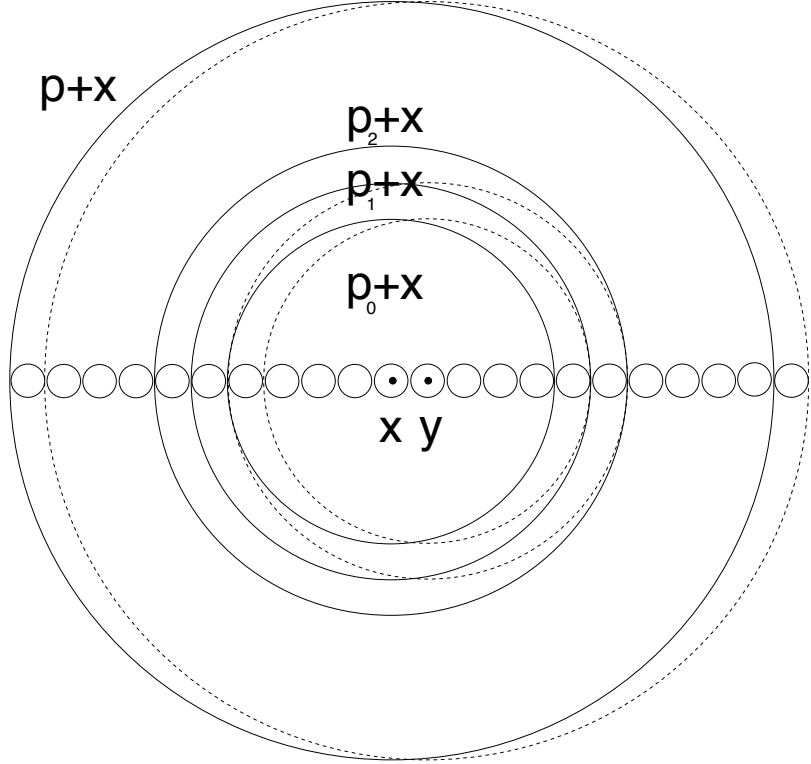}
\caption{The dotted circles, in order of decreasing radii, are the translates $p+y$, $p_1+y$, and $p_0+y$.
The small circles along the horizontal axis are occurrences of the same patch, and illustrate the local periodic pattern generated by the overlaping of $p+x$ and $p+y$.}
\label{fig-nrep}
\end{figure}
%%%
We iterate this argument to obtain that there is a patch $p_j=p_0^{(j)}$ with radius $r_j \le r_0 + 2ja$,  for all $j\le n$.
For $j=n$ this last inequality implies \(n \ge r/(4a)  \ge r_q /(8a) > 1/ (8\epsilon) > m\).
We have thus build a sequence of privileged patches $p_0, p_1, \cdots p_m, \cdots p_n=p$, whose first $m+1$ terms give the sequence in (i).
But $p_0$ is the largest privileged patch in $p$  (with same marker and) of radius $r_0< r/2$, hence $p_1=p_0'$ has radius $r_1\ge r/2= r_n/2 \ge r_m/2$, which proves (ii).
\end{proof}

%%%%%%%%%%%%%%%%%%%%%%%%%%%%%%%%%%%%%%%%%%%%%%%%%%%%%%%%%
\section{The tree of privileged patches and the Connes metrics}
\label{sec-tree}

We build the tree $\Tt$ of privileged patches of $T$ inductively as follows:
%%%
\begin{enumerate}[]

\item{(0)} the root of $\Tt$ stands for the empty patch;

\item{(1)} vertices of order $1$ stand for privileged patches of order 1 (prototiles), each of which is linked by one edge to the root;

\item{(n)} vertices of order $n>1$ stand for privileged patches of order $n$, each of which is linked by one edge to the vertex of order $n-1$ corresponding to the patch it is derived from.% (this edged is well-defined and unique by Lemma~\ref{lem-der} (ii)).

\end{enumerate}
%%%
The tree $\Tt$ is well-defined by Lemma~\ref{lem-der} (ii): each vertex of level $n+1$ is linked to a unique vertex of level $n$, for all $n$.

We let $\partial \Tt$ be the set of infinite rooted path in $\Tt$: $\xi =(\xi_n)_{n\ge 0} \in \partial \Tt$ is a sequence of privileged patches, with $\xi_{n+1}$ derived from $\xi_n$ for all $n$.
Given a vertex $v\in \Tt$, we let $[v] \subset \partial \Tt$ be the {\em cylinder} of $v$, namely the set of all infinite paths through $v$.

%%%
\begin{prop}
\label{prop-tree}
The set $\partial \Tt$ of infinite paths in $\Tt$, endowed with the topology of cylinders, is homeomorphic to the discrete hull $\Xi$.
\end{prop}
%%%
\begin{proof}
The sets $\partial \Tt$ and $\Xi$ are easily seen to be isomorphic.
Given a tiling $T$ in $\Xi$, let $\xi_0$ be the empty patch and $\xi_1$ the prototile occurring in $T$ at the origin.
Since $T$ is repetitive, there is a (unique) privileged patch $\xi_2$ derived from $\xi_1$ which occurs in $T$ at the origin.
We construct inductively a (unique) sequence of privileged patches occurring at the origin of $T$, which defines an infinite path in $\Tt$.
Conversely, by Lemma~\ref{lem-rder} (ii), a sequence of privileged patches in $\partial \Tt$ defines a unique tiling in $\Xi$.

A basis for the topology of $\Xi$ is given by the acceptance domains $[p]$ of patches.
While cylinders correspond to acceptance domains of privileged patches, hence yield a coarser topology on $\Xi$.
Given a patch $p$, let $p_0$ is the greatest privileged patch contained in $p$, and $p_1, \ldots p_k$ the patches derived from $p_0$.
Then $[p] \subset [p_1] \cup \ldots \cup [p_k]$, so both topologies agree.
\end{proof}

A {\em weight function} is any function $\delta: [0,+\infty) \rightarrow (0,1]$, which decreases and has limit $0$ at $+\infty$.
A weight function allows us to defined a ultra metric on $\Xi$, as in equation~\eqref{eq-ultrametric}, and to build a spectral triple on $C(\Xi)$ as explained in Section~\ref{sec-ST}.
The Connes distance of that spectral triple yields two pseudo-metrics on $\partial \Tt \simeq \Xi$, which we now define.

Given $\xi, \xi'\in \partial \Tt$, we let $\xi \wee \xi'$ denote the vertex at which the paths branch in $\Tt$, and $\Oo(\xi \wee \xi')$ the order of that vertex.
If we identify $\xi,\xi'$ with tilings $T,T'\in\Xi$ by Proposition~\ref{prop-tree}, then $\xi\wee\xi'$ is the greatest common privileged patch which occurs in both $T$ and $T'$ at the origin.
The following define two metrics\footnote{$\dinf$ is a ultra-metric, $\dsup$ is valued in $[0,+\infty]$.} on $\partial \Tt$:
%%%
\begin{equation}
\label{eq-dinf}
\dinf (\xi, \xi') = \left\{ \begin{array}{ll} \delta\bigl( r_{\Oo(\xi \wee \xi')}  \bigr)& \text{\rm if } \xi \neq \xi' \\ 0 & \text{\rm if } \xi = \xi' \end{array}\right., 
\end{equation}
%%%
and
%%%
\begin{equation}
\label{eq-dsup}
\dsup(\xi,\xi') = \dinf(\xi,\xi') + \sum_{n> \Oo(\xi \wee \xi')}  \delta(r_n) + \delta(r'_n) \ ,
%\dsup(\xi,\xi') = \dinf(\xi,\xi') + \sum_{n> \Oo(\xi \wee \xi')} \bigl( \delta(r_n) + \delta(r'_n) \bigr),
\end{equation}
%%%
where $r^{(')}_n$ is the radius of the patch $\xi^{(')}_n$ (so one has $r_n=r'_n$ for all $n\le \Oo(\xi \wee \xi')$).

Clearly $\dinf \le \dsup$, and %. As in \cite{KS10,KLS11} 
one easily sees that $\dinf$ and $\dsup$ are Lipschitz equivalent if and only if
%%%
\begin{equation}
\label{eq-Lip}
\exists C>0, \, \forall \xi \in \partial \Tt, \, \forall m\in \NM, \qquad
\delta(r_m)^{-1} \sum_{k \ge 1}  \delta(r_{m+k}) \le C \,.
\end{equation}
%%%

%%%%%%%%%%%%%%%%%%%%%%%%%%%%%%%%%%%%%%%%%%%%%%%%%%%%%%%%%
\section{Characterisation of repulsive tilings}
\label{sec-charact}

We state our main result.
Let $T$ be an aperiodic, repetitive, and FLC tiling of $\RM^d$, as in Definition~\ref{def-hypT}.
Consider the tree $\Tt$ of privileged patches of $T$ as in the previous section.
Let  \(\delta: [0,+\infty) \rightarrow (0,1]\) be a weight function as in the previous section (decreasing with limit $0$ at infinity), and assume that there exists $c_1, c_2>0$ such that
%%%
\begin{equation}
\label{eq-delta}
\delta(ab) \le c_1 \delta(a) \delta(b), \qquad \delta(2a)\ge c_2 \delta(a), \ \forall a,b \ge 0.
\end{equation}
%%%
Consider the metrics $\dinf$ and $\dsup$ on $\partial \Tt$ given in equations~\eqref{eq-dinf} and~\eqref{eq-dsup}.

%%%%%%%
\begin{thm}
\label{thm-main}
%Assume that the weight function satisfies: for any $a,b \in \RM^+$ large enough, $\delta(ab) \le c_1 \delta(a) \delta(b)$ and $\delta(2a)\ge c_2 \delta(a)$, for some $c_1,c_2>0$.
%$T$ is repulsive if and only if $\dinf$ and $\dsup$ are Lipschitz equivalent.
The following are equivalent:
%%%
\begin{enumerate}[(i)]
\item $T$ is repulsive,
\item $\dinf$ and $\dsup$ are Lipschitz equivalent.
\end{enumerate}
%%%
\end{thm}
%%%%%%%
\begin{proof}
Assume $T$ is repulsive, so $\ell>0$ in equation~\eqref{eq-rep}.
Upon rescalling $\delta$, we may assume that  $c_1=1$ in equation~\eqref{eq-delta}, and $\delta(2\ell +1)<1$.
Pick $m\in\NM$ and $\xi\in\partial \Tt$.
Let $r_n$ be the radius of the patch $\xi_n$, $n\ge 0$.
By Lemma~\ref{lem-rder} (ii), for any $k\ge 1$ we have $r_{m+k}\ge (2\ell +1)^k r_m$.
Hence
%%%
\[
\delta(r_{m})^{-1} \sum_{k\ge 1 }  \delta(r_{m+k}) \le  
\delta(r_{m})^{-1} \sum_{k\ge 1 }  \delta((2\ell+1)^k) \delta(r_{m}) \le
 \sum_{k\ge 1 }  \delta(2\ell +1)^k\,,
\]
%%%
where the last two inequalities follows from equation~\eqref{eq-delta}.
The converging geometric series on the right hand side gives a uniform bound in equation~\eqref{eq-Lip}, which proves that $\dinf$ and $\dsup$ are Lipschitz equivalent.

Assume that $T$ is not repulsive.
Fix an integer $N$ (large), and consider a sequence of privileged patches $p_0, p_1, \ldots p_{m}$, $m>N$, as in Lemma~\ref{lem-nrep}.
Choose an infinite path $\xi\in\partial \Tt$ going through the vertices associated with $p_0, p_1 \ldots p_{m}$.
Upon a change of index, we may assume that $\xi_j$ corresponds to $p_j$, for $j=1, \ldots m$.
Then
%%%
\[
\delta(r_{1})^{-1} \sum_{k\ge 1}  \delta(r_{1+k}) \ge \frac{1}{\delta(r_{1})} \sum_{j=2}^{m} \delta(r_j) \ge \frac{m-1}{\delta(r_{1})}\delta(r_m) \ge \frac{N}{\delta(r_1)} \delta(2r_1)\ge Nc_2\,,
\]
%%%
where we used that $\delta$ decreases, and equation~\eqref{eq-delta}.
Since $N$ was arbitrary, one cannot bound the series on the left hand side.
So there exists no uniform bound in equation~\eqref{eq-Lip}, and thus $\dinf$ and $\dsup$ are not Lipschitz equivalent.
\end{proof}

%%%%%%%%%%%%%%%%%%%%%%%%%%%%%%%%%%%%%%%%%%%%%%%%%%%%%%%%%
\section{The spectral triple}
\label{sec-ST}

For the sake of completeness, we remind the reader of the spectral triple on $C(\Xi)\cong C(\partial \Tt)$ whose Connes distance yields $\dinf$ and $\dsup$.
The construction in \cite{KS10} is given for any tree, and in \cite{KLS11} for the tree of privileged words of a $1d$-subshift, which we rewrite here for the tree of privileged patches defined in Section~\ref{sec-tree}.
These constructions are related to other spectral triples build for metric spaces \cite{Ri99,Ri04,CI07} or more specifically fractals
\cite{GI03,GI05,CIL08} and ultrametric Cantor sets \cite{PB09}.
We refer the reader to \cite{KS10} and \cite{KLS11} for details and proofs.

We consider the tree $\Tt=(\Tt^0,\Tt^1)$ of privileged patches, and a weight $\delta$ as in Section~\ref{sec-tree}.
We add {\em horizontal edges} $\Hh$ to the graph $\Tt$: $\Hh=\cup_{n\ge 1} \Hh_n$, and $\Hh_n$ is a set of oriented edges between vertices of level $n$ in $\Tt$ defined as follows. If $u_1,u_2\in \Tt$, then there is one horizontal edge $h\in \Hh_n$ with source $s(h)=u_1$ and range $r(h)=u_2$, if and only if $u_1$ and $u_2$ stand for two distinct privileged patches of order $n$ both of which are derived from the same privileged patch of order $n-1$.
Given any such $h$, there is then an edge $h^\op\in \Hh$ with reverse orientation: $r(h^\op)=s(h)$ and $s(h^\op)=r(h)$, and $(h^\op)^\op=h$.
We fix an orientation of $\Hh$, and write the decomposition into positively and negatively oriented edges $\Hh=\Hh_+ \cup \Hh_- $.

A {\em choice} is a map $\tau: \Tt^{0} \rightarrow \partial \Tt$ such that $\tau(v)$ is an infinite path through vertex $v$.
The {\em approximation graph} $G_\tau=(V,E)$ is defined by
\[
V= \tau(\Tt^0), \qquad E= \tau\times \tau (\Hh).
\]
The orientation on $\Hh$ is carried over to $E=E_+ \cup E_-$.
We endow $G_\tau$ with the graph metric induced by the weight $\delta$: for $e=\tau\times\tau(h) \in E$ we set the length of $e$ to be $\ell(e) = \delta(r_h)$, where $r_h$ is the radius of the privileged patch from which $s(h)$ and $r(h)$ are derived.
The set of vertices $V$ is dense in $\partial \Tt$, and the set of edges $E$ encodes adjacencies according to the choice $\tau$.

We consider the spectral triple associated with the approximation graph $G_\tau=(V,E)$: $\bigl( C(\partial\Tt), \ell^2(E), \pi_\tau, D\bigr)$.
The C$^\ast$-algebra $C(\partial \Tt)$ of continuous functions over $\partial\Tt$ is faithfully represented on the Hilbert space $\ell^2(E)$ by
\[
\pi_\tau (f) \phi(e) = f\bigl( s(e) \bigr) \phi(e).
\]
The Dirac $D$ is the unbounded operator on $\ell^2(E)$, with compact resolvent, given by
\[
D\phi (e) = \frac{1}{\ell(e)} \phi( e^\op ).
\]
The ``non commutative derivation'' is the finite difference operator 
\[
\bigl[ D,\pi_\tau(f) \bigr] \phi(e) = \frac{f(s(e)) - f(r(e))}{\ell(e)} \phi (e^\op),
\]
which is bounded over the pre-C$\ast$-algebra $C_{Lip}(\partial \Tt)$ of Lipschitz continuous functions over $\partial \Tt$.
The Connes distance of the spectral triple is a pseudo-metric over $\partial \Tt$ which reads
\begin{eqnarray*}
d_\tau (\xi, \xi') &=&  \sup_{f\in C(\partial \Tt)} \Bigl\{ \bigl| f(\xi) - f(\xi')\bigr| : \| [D, \pi_\tau(f) \| \le 1 \Bigr\}  \\
& = & \sup_{f\in C(\partial \Tt)} \Bigl\{ \bigl| f(\xi) - f(\xi') \bigr| : |f(s(e)) - f(r(e))| \le \ell(e), \, \forall e\in E \Bigr\}
\end{eqnarray*}
where the norm is the operator norm on $\ell^2(E)$.
It is an extension of the graph metric of $G_\tau$ to $\partial \Tt$.
If it is continuous for the topology of $\Tt$, it reads explicitly:
\[
d_\tau(\xi, \xi') = \dinf(\xi, \xi') + \sum_{n> \Oo(\xi\wee\xi')} b_\tau(\xi_n) \delta(r_n) + b_\tau(\xi'_n) \delta(r'_n)\,,
\]
where $\xi^{(')}=(\xi^{(')}_n)$, $r^{(')}_n$ is the radius of the patch $\xi^{(')}_n$, and for $\eta\in\partial\Tt$, $b_\tau(\eta_n) = 1$ if $\eta$ and $\tau(\eta_n)$ branch at $\eta_n$, and $b_\tau(\eta_n) = 0$ otherwise.
The infimum and supremum of $d_\tau$ over choices $\tau$ yield the metrics $\dinf$ of equation~\eqref{eq-dinf} and $\dsup$ and~\eqref{eq-dsup} respectively.


\begin{thebibliography}{BEL}


%\bibitem{BBL13} A. Besbes, M. Boshernitzan, D. Lenz. ``Delone sets  with finite local complexity: Linear repetitivity versus positivity of weights''. {\em Discrete and Computat. Geom.}  {\bf 49} (2013) 335-347.

\bibitem{Co94} A. Connes. {\it Noncommutative geometry}. Academic
 Press Inc., San Diego, CA, 1994.

\bibitem{CI07} E. Christensen, C. Ivan.
``Sums of two-dimensional spectral triples''.  {\em Math. Scand.}
{\bf 100} (2007) 35--60.

\bibitem{CIL08} E. Christensen, C. Ivan, M.L. Lapidus.
``Dirac operators and spectral triples for some fractal sets built on
curves''. {\it Adv. Math.} {\bf 217} (2008), no. 1, 42--78.


\bibitem{Dur00} F. Durand. ``Linearly recurrent subshifts have a finite number of non-periodic subshift factors''. {\it Ergodic Theory Dynam. Systems} {\bf 20} (2000) 1061--1078.

\bibitem{Dur03} F. Durand. Corrigendum and addendum to \cite{Dur00}.
%``Linearly recurrent subshifts have a finite number of non-periodic subshift factors'' (reference \cite{Dur00} above).
%[{\it Ergodic Theory Dynam. Systems} {\bf 20} (2000) 1061--1078].
{\it Ergodic Theory Dynam. Systems} {\bf 23} (2003) 663--669.


\bibitem{GJWZ09} A. Glen, J. Justin, S. Widmer, L.Q. Zamboni. ``Palindromic richness''. {\it European J. Combin.} {\bf 30} (2009) 510--531.


\bibitem{GI03} D. Guido, T. Isola. ``Dimensions and singular traces for spectral triples, with applications for fractals''. {\it J. Func. Anal.} {\bf 203} (2003) 362--400.

\bibitem{GI05} D. Guido, T. Isola. ``Dimension and spectral triples for fractals in $\RM^n$''. In Advances in Operator Algebras and Mathematical Physics, {\it Theta Ser. Adv. Math.} {\bf 5}, Theta, Bucarest (2005), 89--108.


\bibitem{FJS13} M. Forsyth, A. Jayakumar, J. Shallit. ``Remarks on Privileged Words''. {\em Eprint}  arXiv:1306.6768 (math.CO).

\bibitem{FPZ13} A.E. Frid, S. Puzynina, L. Zamboni. ``On palindromic factorization of words''. {\it Adv. in Appl. Math.} {\bf 50} (2013) 737-748.

\bibitem{KS10} J. Kellendonk, J. Savinien. ``Spectral triples and characterization of aperiodic order''. {\it Proc. London Math. Soc.} {\bf 104} (2012) 123-157.

\bibitem{KLS11} J. Kellendonk, D. Lenz, J. Savinien. ``A characterisation of subshifts with bounded powers''. {\it Discrete Math.} {\bf 313} (2013) 2881-2894.

\bibitem{LP03} J. Lagarias, P. Pleasants. ``Repetitive Delone sets and  quasicrystals''. {\it Ergodic Theory Dynam. Systems} {\bf 23} (2003) 831--867.

\bibitem{Len04} D. Lenz. ``Aperiodic linearly repetitive delone sets are densely repetitive''. {\it Discrete Comput. Geom.} {\bf 31} (2004) 323-326.

\bibitem{Pat98}  J. Patera (ed.): {\it Quasicrystals and Discrete Geometry}, Toronto, ON, 1995. Fields Inst. Monogr., vol.
10. Am. Math. Soc., Providence (1998)


\bibitem{PB09} J. Pearson, J. Bellissard. ``Noncommutative Riemannian geometry and diffusion on ultrametric Cantor sets''. {\it J. Noncommut. Geom.} {\bf 3} (2009) 447--481.


\bibitem{Pelto13} J. Peltom\"aki. ``Introducing Privileged Words: Privileged Complexity of Sturmian Words''. {\it  Theoret. Comput. Sci.} {\bf 500} (2013) 57-67.

\bibitem{Pelto} J.  Peltom\"aki. ``Privileged Factors in the Thue-Morse Word - A Comparison of Privileged Words and Palindromes''. {\em Eprint}  arXiv:1306.6768 (math.CO).

\bibitem{Ri99} M. Rieffel.
``Metrics on state spaces''.  {\em Doc. Math.}  {\bf 4}  (1999) 559--600.

\bibitem{Ri04} M. Rieffel. ``Compact Quantum Metric Spaces''. {\it Operator algebras, quantization, and noncommutative geometry} 315--330, Contemp. Math. {\bf 365}, Amer. Math. Soc., Providence (2004).


\end{thebibliography}
\end{document}